\documentclass[11pt]{article}

\usepackage{amsfonts,amssymb}
\usepackage{diagrams}

\usepackage[latin1]{inputenc}  
\usepackage[T1]{fontenc}

\def\pplus{\tilde{+}}%% A MODIFIER !!!!!!!!!!!!!!

\newcounter{paragrafsubsub}[subsubsection]
\renewcommand{\theparagrafsubsub}{%
\thesubsubsection.\roman{paragrafsubsub}}
\newcommand{\paragrafsubsub}{%
\refstepcounter{paragrafsubsub}
{\bf \theparagrafsubsub}\hspace{0.2em}--- }

\newcounter{paragrafsub}[subsection]
\renewcommand{\theparagrafsub}{\thesubsection.\arabic{paragrafsub}}
\newcommand{\paragrafsub}{%
\refstepcounter{paragrafsub}
{\bf \theparagrafsub}\hspace{0.2em}--- }

\newcounter{paragraf}[section]
\renewcommand{\theparagraf}{\thesection.\arabic{paragraf}}
\newcommand{\paragraf}{%
\refstepcounter{paragraf}
{\bf \theparagraf}\hspace{0.2em}--- }

\newcommand\paragraphe{%
\par \indent
\ifcase\value{subsection} %
\paragraf
\else
\ifcase\value{subsubsection}\paragrafsub %
\else\paragrafsubsub
\fi\fi
}

\def\longto{\longrightarrow}

\def\Id{{\rm Id}}\def\Ker{{\rm Ker}}
\def\eps{\varepsilon}
\def\wbar{\overline{w}}
\def\Rep{{\rm Rep}}
\def\Hom{{\rm Hom}}
\def\GL{{\rm GL}}\def\SL{{\rm SL}}
\def\ext{{\rm ext}}

\def\Det{{\mathcal Det}}
 \def\Tau{{\mathcal T}}

\def\NN{{\mathbb N}}
\def\ZZ{{\mathbb Z}}

\def\kk{{\mathbb K}}

\def\lieg{{\mathfrak g}}
\def\lieu{{\mathfrak u}}
\def\lp{{\mathfrak p}} \def\lb{{\mathfrak b}}

\def\kbprod{\odot_0}

\def\GL{{\rm GL}}
\def\codim{{\rm codim}}

\newtheorem{lemma}{Lemma}

\newtheorem{theo}{Theorem}
\newtheorem{coro}{Corollary}

\newtheorem{theoi}{Theorem}

\newenvironment{proof}{{\noindent\bf Proof.}}{\hfill $\square$}
\newenvironment{defin}{{\noindent\bf Definition.}}{\\}
\newenvironment{remark}{{\noindent\bf Remark.}}{}

\begin{document}
\title{Multiplicative formulas in Cohomology of $G/P$\\
and in quiver representations}
\author{N. Ressayre\footnote{ {\tt ressayre@math.univ-montp2.fr}}}

\maketitle
\begin{abstract}

\end{abstract}

\section{Introduction}

Consider a partial flag variety $X$ which is not a grassmaninan.
Consider also its cohomology ring ${\rm H}^*(X,\ZZ)$ endowed with the base
formed by the Poincaré dual classes of the Schubert varieties.
In \cite{Richmond:recursion}, E.~Richmond showed that some coefficient
structure of the product in ${\rm H}^*(X,\ZZ)$ are products of two such coefficients for 
smaller flag varieties.

Consider now a quiver without oriented cycle. If $\alpha$ and $\beta$ denote two
dimension-vectors, $\alpha\circ\beta$ denotes the number of $\alpha$-dimensional 
subrepresentations of a general $\alpha+\beta$-dimensional representation. 
In \cite{DW:comb}, H.~Derksen and J.~Weyman expressed some numbers $\alpha\circ\beta$ as products 
of two smaller such numbers.

The aim of this work is to prove two generalisations of the two above results by the same way.\\

We now explain our result about cohomology of the generalized flag varieties.
Let $G$ be a semi-simple group, $T\subset B\subset Q\subset G$ be a maximal torus,
a Borel subgroup and a parabolic subgroup respectively.
In \cite{BK}, P.~Belkale and S.~Kumar  defined a new product (associative and commutative) on the 
cohomology group ${\rm H}^*(G/Q,\ZZ)$ denoted by $\kbprod$.  
Any coefficient structure of $\kbprod$ in the base of Schubert classes is either zero or the 
corresponding  coefficient structure for the cup product.

Let now $P\subset Q$ be second parabolic subgroups of $G$ and $L$ denote the Levi subgroup 
of $P$ containing $T$. We obtain the following:

\begin{theoi}
\label{thi:kbprod}
  Any coefficient structure of $({\rm H}^*(G/Q,\ZZ),\kbprod)$ in the base of Schubert classes is
the product of such two coefficients for $({\rm H}^*(G/P,\ZZ),\kbprod)$ and 
$({\rm H}^*(L/(L\cap Q),\ZZ),\kbprod)$ respectively.
\end{theoi}

Actualy, Theorem~\ref{th:kbprod} below is more precise and explicit than Theorem~\ref{thi:kbprod}.
This result was already obtained in \cite{Richmond:recursion} when $G=\SL_n$, $Q$ is any parabolic 
subgroup and $P$ is the maximal parabolic subgroup corresponding to the linear subspace in $G/Q$
of minimal dimension.\\

Let $Q$ be a quiver.
The Ringle form (see Section~\ref{sec:defcarquois}) is denoted by 
$\langle\cdot,\cdot\rangle$. We prove the following:

\begin{theoi}
\label{thi:carquois}
  Let $\alpha,\,\beta$ and $\gamma$ be three dimension-vectors. 
We assume that $\langle\alpha,\beta\rangle=\langle\alpha,\gamma\rangle=
\langle\beta 
,\gamma\rangle=0$.
Then, 
$$
(\alpha+\beta\circ \gamma).(\alpha\circ\beta)= 
(\alpha\circ \beta+\gamma).(\beta\circ\gamma).
$$
\end{theoi}

Note that, Theorem~\ref{th:carquois} is more general than Theorem~\ref{thi:carquois}, since
$s$  dimension-vectors occur. Moreover, we give in  Theorem~\ref{thi:carquois} a geometric 
interpretation of the product.

If $Q$ has no oriented cycle, we obtain the following corollary:
\begin{theoi}
\label{thi:dw}
We assume that $Q$ has no oriented cycle.
  Let $\alpha,\,\beta$ and $\gamma$ be three dimension-vectors. 
We assume that $\langle\alpha,\beta\rangle=\langle\alpha,\gamma\rangle=0$ and 
$\beta\circ\gamma=1$.

Then, $\alpha\circ \beta+\gamma=(\alpha\circ\beta).(\alpha\circ \gamma)$.  
\end{theoi}

This result is not really stated in \cite{DW:comb}. 
However, the proof of \cite[Theorem~7.14]{DW:comb} implies it. 
Note that the proof of  Theorem~\ref{thi:carquois} is really diffenrent from those of 
\cite[Theorem~7.14]{DW:comb}. Indeed, the numbers $\alpha\circ\beta$ have two non trivially
equivalent interpretations (see \cite{DSW:nbsubrep}): as a number of points in a generic fiber of 
a morphism or as a dimension of the subspace of invariant vectors in a representation. Here we use
the first characterisation. Derksen and Weyman used the second one.\\

In Section~\ref{sec:gen}, we consider more generally a semi-simple group $G$ acting on a variety 
$X$. 

\section{Degree of dominant pairs}
\label{sec:gen}

\subsection{Definitions}
\label{sec:def}

Let $G$ be a reductive group acting on a smooth variety $X$.
Let $\lambda$ be a one parameter subgroup of $G$.
Let $G^\lambda$ or $L$ denote the centralizer of $\lambda$ in $G$.
We consider the usual 
parabolic subgroup associated to $\lambda$ with Levi subgroup $L$:
$$
P(\lambda)=\left \{
g\in G \::\:
\lim_{t\to 0}\lambda(t).g.\lambda(t)^{-1} 
\mbox{  exists in } G \right \}.
$$

Let $C$ be an irreducible  component of the fix point set $X^\lambda$ of $\lambda$ in $G$.
We also consider the Bialinicky-Birula cell $C^+$ associated to $C$:
$$
C^+=\{x\in X\ |\ \lim_{t\to 0}\lambda(t)x\in C\}.
$$
Then, $C$ is stable by the action of $L$ and $C^+$ by the action of $P(\lambda)$.

Consider over $G\times C^+$ the action of $G\times P(\lambda)$ given by the formula 
(with obvious notation): $(g,p).(g',y)=(gg'p^{-1},py)$.
Consider the  quotient $G\times_{P(\lambda)}C^+$ of $G\times C^+$ by the action of 
$\{e\}\times P(\lambda)$.
The class of a pair $(g,y)\in G\times C^+$ in $G\times_{P(\lambda)} C^+$ is denoted by $[g:y]$.

The action of $G\times\{e\}$ induces an action of $G$ on $G\times_{P(\lambda)}C^+$.
Moreover, the first projection $G\times C^+\longto G$ induces a $G$-equivariant map  
$\pi\,:\,G\times_{P(\lambda)} C^+\longto G/P(\lambda)$ which is a locally trivial fibration 
with fiber $C^+$. In particular, we have
$$
\dim(G\times_{P(\lambda)}C^+)=\dim(G/P(\lambda))+\dim(C^+).
$$

Consider also the $G$-equivariant map $\eta\,:\,G\times_{P(\lambda)}C^+\longto X,\,[g:y]\mapsto gy$.
We finally obtain:
\begin{diagram}
  G\times_{P(\lambda)}C^+&\rTo^\eta&X.\\
\dTo^\pi\\
G/P(\lambda)
\end{diagram}

It is well known that
the map
\begin{eqnarray}
  \label{eq:immersion}
\begin{array}{ccc}
G\times_{P(\lambda)}C^+&\longto& G/{P(\lambda)}\times X \\
\left[ g:y\right] & \longmapsto& (g{P(\lambda)},gy)
\end{array}
\end{eqnarray}

is an immersion; its image is the set of the $(g{P(\lambda)},x)\in G/{P(\lambda)}\times X$ such that $g^{-1}x\in Y$.
Note that this fact can be used to prove that  $G\times_{P(\lambda)} C^+$ actually exists.\\

\begin{defin}
  We set $$
  \begin{array}{ll}
    \delta(G,X,C,\lambda)&=\dim(X)-\dim(G/P(\lambda))-\dim(C^+)\\
    &=\codim(C^+,X)-\codim(P(\lambda),G).
\end{array}$$
If $\delta(G,X,C,\lambda)=0$ and $\eta$ is dominant, it induces a finite field extension:  
$k(X)\subset k(G\times_{P(\lambda)}C^+)$. We denote by $d(G,X,C,\lambda)$ the degree of this extension.
If $\delta(G,X,C,\lambda) \neq 0$ or $\eta$ is not dominant, we set $d(G,X,C,\lambda)=0$.

More generally, we define the degree of any morphism to be the degree of the induced extension 
if it is finite and zero otherwise. 
\end{defin}

\subsection{A product formula for $d(G,X,C,\lambda)$}

Let $T$ be a maximal torus of $G$ and $x_0$ be a fixed point of $T$ in $X$.
We keep notation of Section~\ref{sec:def} and assume that the image of $\lambda$ is contained in $T$
and $x_0\in C$. We set $P=P(\lambda)$.

Let $\lambda_\eps$ be another one parameter subgroup of $T$. 
Set $P_\eps=P(\lambda_\eps)$. Consider the irreducible component $C_\eps$ of 
$X^{\lambda_\eps}$ which contains $x_0$ and $C_\eps^+=\{x\in X\,:\,\lim_{t\to 0}\lambda_\eps(t)x\in C\}$.
We assume that:
\begin{enumerate}
\item $P_\eps\subset P$,
\item $C^+_\eps\subset C^+$, and
\item $C_\eps\subset C$.
\end{enumerate}

\begin{remark}
  Notice that the set of the $\lambda_\eps$ which satisfy these three assumptions generated an open convex cone
in the vector space containing the one parameters subgroups of $T$ as a lattice.
\end{remark}

Now, we want to compare $\eta$ and $\eta_\eps$. We introduce the natural morphism:
$$
\eta_L\,:\,L\times_{P_\eps\cap L}(C^+_\eps\cap C)\longto C.
$$
This map is a map $\eta$ as in Section~\ref{sec:def} with $G=L$, $X=C$, $C=C_\eps$ and
$\lambda=\lambda_\eps$. In particular, we have defined 
$\delta(L,C,C^+_\eps \cap C,\lambda_\eps)$ and $d(L,C,C^+_\eps \cap C,\lambda_\eps)$. 

We can now state our main result

\begin{theo}
\label{th:ppal}
  With above notation, we have:
  \begin{enumerate}
  \item $\delta(G,X,C_\eps,\lambda_\eps)=
\delta(L,C,C_\eps,\lambda_\eps)+\delta(G,X,C,\lambda)$, and
  \item If $\delta(L,C,C_\eps,\lambda_\eps)=\delta(G,X,C,\lambda)=0$, then 
$$d(G,X,C_\eps,\lambda_\eps)=d(L,C,C_\eps,\lambda_\eps)\cdot d(G,X,C,\lambda).$$
  \end{enumerate}
\end{theo}

\subsection{Proof of Theorem~\ref{th:ppal}}

\paragraphe
We set 
$$
Y_L=L\times_{P_\eps\cap L}(C^+_\eps\cap C) {\rm\ \ \ and\ \ \ }
Y_P=P\times_{P_\eps}C_\eps^+.
$$
We consider the natural morphism
$$
\eta_P\,:\,Y_P\longto C^+,
$$ 
and
$$
[\Id:\eta_P]\,:\,G\times_P Y_P
\longto G\times_PC^+,\; [g:[p:x]]\longmapsto[g:px].
$$

\begin{lemma}
\label{lem:iso1}
With above notation, we have:
\begin{enumerate}
\item the map $G\times_P Y_P\longto G\times_{P_\eps}C_\eps^+$,
$[g:[p:x]]\longmapsto[gp:x]$ is an isomorphism denoted by $\iota$; moreover, 
\item $\eta_\eps\circ\iota=\eta\circ ([\Id:\eta_P])$.
\end{enumerate}
\end{lemma}

\begin{proof}
  The morphism $\iota$ commutes with the two projections on $G/P$. Moreover, the restriction 
of $\iota$ over $P/P$ is the closed immersion $P\times_{P_\eps}C_\eps^+\longto G\times_{P_\eps}C_\eps^+$.
It follows (see for example \cite[Appendice]{GammaGH}) that $\iota$ is an isomorphism.

The morphisms $\eta_\eps\circ\iota$ and $\eta\circ ([\Id:\eta_P])$ are $G$-equivariant and extend the 
immersion of $C_\eps^+$ in $X$. They have to be equal.
\end{proof}\\

\paragraphe
We are now interested in $\eta_P$.
Consider the two following limit morphisms:
$$
\begin{array}{lccl}
  \Lambda_P\,:&P&\longto&L\\
&p&\mapsto&\lim_{t\to 0}\lambda(t)p\lambda(t^{-1})
\end{array}
\begin{array}{c}
  {\rm \ and}\\\ 
\end{array}
\begin{array}{lccl}
  \Lambda^+\,:&C^+&\longto&C\\
&y&\mapsto&\lim_{t\to 0}\lambda(t)y.
\end{array}
$$

The computation $\lambda(t)px=\lambda(t)p\lambda(t^{-1})\lambda(t)x$ implies the easy
\begin{lemma}
\label{lem:Lambdaequiv}
We have: $\Lambda^+(px)=\Lambda_P(p)\Lambda^+(x)$.  
\end{lemma}

\paragraphe
Recall that $\Lambda^+\,:\,C^+\longto C$ is a vector bundle. The pullback of this vector bundle
by $\eta_L$ is 
$$\eta^*_L(C^+)=\{([l:x],y)\in Y_L\times C^+\;|\;lx=\Lambda^+(y)\},$$
endowed with the first projection $p_1$ on $Y_L$.
Consider the following diagram:

\begin{eqnarray}
  \label{eq:diag}
\begin{diagram}
  Y_P&&\rTo^{\Theta\,:\,[p:x]\mapsto ([\Lambda_P(p):\Lambda^+(x)],px)\;}&&\eta_L^*(C^+)\\
&\rdTo_{[p:x]\mapsto [\Lambda_P(p):\Lambda^+(x)]}&&\ldTo_{p_1}&\\
&&Y_L&&\\
&&\dTo&&\\
&&L/(P_\eps\cap L).&&
\end{diagram}
\end{eqnarray}

\begin{lemma}
\label{lem:thetaiso}
The above diagram is commutative, and the top horizontal map $\Theta$ is an isomorphism.
\end{lemma}

\begin{proof}
  First, note that the map $Y_P\longto Y_L$ in Diagram~\ref{eq:diag} is well defined by Lemma~\ref{lem:Lambdaequiv}.
Diagram~\ref{eq:diag} is obviously commutative.

Since all the morphisms in Diagram~\ref{eq:diag} are $L$-equivariant, \cite[Appendice]{GammaGH} implies that it is sufficient
to prove that $\Theta$ is an isomorphism when restricted over the class of $e$ in $L/(P_\eps\cap L)$.
The fiber in $Y_L$ over this point if $C\cap C_\eps^+$.
Since $P^u\subset P_\eps^u$, the fiber in $Y_P$ identify with $C_\eps^+$, by $x\in C_\eps^+\mapsto [e:x]$.
The fiber in $\eta_L^*(C^+)$ also identify with $C_\eps^+$ in such a way the restriction of $\Theta$ becomes the identity.
It follows that $\Theta$ is an isomorphism.
\end{proof}\\

\paragraphe
We can now prove Theorem~\ref{th:ppal}.

\begin{proof}
By Lemma~\ref{lem:thetaiso}, we have the following commutative diagram:

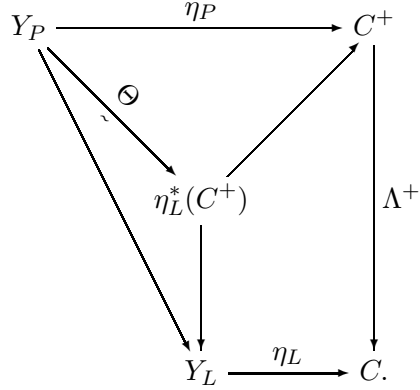
\begin{diagram}
  Y_P&&\rTo^{\eta_P}&&C^+\\
&\rdTo_{\tilde{}}^\Theta\rdTo(2,4)&&\ruTo&\\
&&\eta_L^*(C^+)&&\dTo_{\Lambda^+}\\
&&\dTo\\
&&Y_L&\rTo^{\eta_L}&C.
\end{diagram}
It follows that $\dim(C^+)-\dim(Y_P)=\delta(L,C,C_\eps,\lambda_\eps)$ and 
$d(L,C,C_\eps,\lambda_\eps)$ equals the degree of $\eta_P$.

Moreover, by Lemma~\ref{lem:iso1}, we have the following commutative diagram:

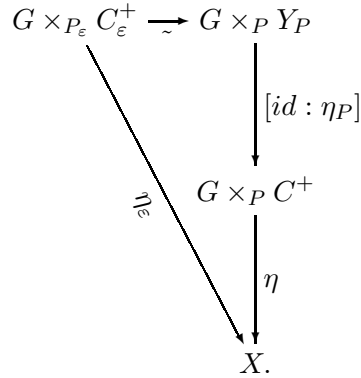
\begin{diagram}
  G\times_{P_\eps}C_\eps^+&\rTo_{\tilde{}}&G\times_P Y_P\\
&\rdTo(2,4)_{\eta_\eps}&\dTo_{[id:\eta_P]}\\
&&G\times_PC^+\\
&&\dTo_\eta\\
&&X.
\end{diagram}
The first assertion follows immediately. Let $d$ denote the degree of $[id:\eta_P]$ that is the degree of
$\eta_P$. Since $d=d(L,C,C_\eps,\lambda_\eps)$, we have to prove that $d(G,X,C_\eps,\lambda_\eps)=d.d(G,X,C,\lambda)$.
We firstly assume that  $d(G,X,C_\eps,\lambda_\eps)=0$.
Since $\delta((G,X,C_\eps,\lambda_\eps)=0$, $\eta_\eps$ is not dominant. 
So, $\eta$ or $[id:\eta_P]$ is not dominant. It follows that 
either $d(G,X,C,\lambda)$ or $d$  is zero. The assertion follows.

We now assume that $d(G,X,C_\eps,\lambda_\eps)\neq 0$, that is that $\eta_\eps$ is dominant.
Since the image of $\eta_\eps$ is contained in the image of $\eta$, $\eta$ is dominant. 
Since $\eta_\eps$ is dominant, the dimension of the closure of the image of $[id:\eta_P]$ at least those of $X$.
Since $\delta(L,C,C_\eps,\lambda_\eps)=\delta(G,X,C,\lambda)=0$, this implies that $\eta_P$ is dominant.
Now, the second assertion is simply the multiplicative formula for the degree of a double extension field.
\end{proof}

\subsection{Well genericaly finite pairs}

\paragraphe
If $Y$ is a smooth variety of dimension $n$, $\Tau Y$ denotes its tangent bundle.
The line bundle $\bigwedge^n\Tau Y$ over $Y$ will be called the {\it determinant bundle} and
denoted by $\Det Y$.
If $\varphi\::\:Y\longto Y'$ is a morphism between smooth variety, we denote
by $T\varphi\::\:\Tau Y\longto\Tau Y'$ its tangent map, and by 
$\Det\varphi\::\:\Det Y\longto\Det Y'$ its determinant.

\paragraphe
Consider $\eta\,:\,G\times_{P(\lambda)}C^+\longto X$ as in Section~\ref{sec:def}. 

\begin{defin}
We say that $(G,X,C,\lambda)$ is {\it genericaly finite} if $d(G,X,C,\lambda)\neq 0$.
We say that $(G,X,C,\lambda)$ is {\it well genericaly finite} if it is genericaly finite and there exists $x\in C$ such 
that $T\eta_{[e:x]}$ is invertible.
\end{defin}

\paragraphe
Consider the restriction of $T\eta$ to $C$:
$$
T\eta_{|C^+}\::\:\Tau(G\times_PC^+)_{|C^+}
\longto \Tau(X)_{|C^+},
$$
and the restriction of $\Det\eta$ to $C^+$:
$$
\Det\eta_{|C^+}\::\:\Det(G\times_P C^+)_{|C^+}
\longto \Det(X)_{|C^+}.
$$
Since $\eta$ is $G$-equivariant, the morphism $\Det\eta_{|C^+}$ is $P$-equivariant;
it can be thought as a $P$-invariant section of the line bundle 
${\mathcal D}:=\Det(G\times_P C^+)_{|C^+}^*\otimes\Det(X)_{|C^+}$ over $C^+$.
For any $x\in C$, $\kk^*$ acts linearly via $\lambda$ on the fiber ${\mathcal D}_x$ over $x$ in ${\mathcal D}$:
this action is given by a character of $\kk^*$, that is an interger $m$.
Moreover, this integer does not depends on $x$ in $C$:
we denote by $\mu^{\mathcal D}(C,\lambda)$ this interger.

\begin{lemma}
  We assuma that $X$ is smooth. The, the following are equivalent:
  \begin{enumerate}
  \item $(G,X,C,\lambda)$ is well genericaly finite;
\item $(G,X,C,\lambda)$ is genericaly finite and  $\mu^{\mathcal D}(C,\lambda)=0$.
  \end{enumerate}
\end{lemma}

\begin{proof}
  Let us assume that $(G,X,C,\lambda)$ is well genericaly finite and let $x\in C$ be such that $T\eta_x$ is invertible.
Then, $\Det\eta_x$ is a non zero $\kk^*$-fixed point in ${\mathcal D}_x$: the action of $\kk^*$ on the line
${\mathcal D}_x$ must be trivial.

Let us now assume that  $(G,X,C,\lambda)$ is genericaly finite and  $\mu^{\mathcal D}(C,\lambda)=0$.
Since the base field is assumed to be of characteristic zero, the exists a point in $G\times_{P(\lambda)}C^+$ where the 
$T\eta$ is invertible. 
Since $\eta$ is $G$-equivariant, one can find such a point $y$ in $C^+$.
In particular, $\Det\eta_{|C^+}$ is a non zero $P(\lambda)$-invariant section of ${\mathcal D}$.
Since  $\mu^{\mathcal D}(C,\lambda)=0$, \cite[Proposition~5]{GITEigen} implies that $\Det\eta_{|C}$
is non identacaly zero.
\end{proof}

\paragraphe
The well genericaly finite pairs provide a nice standing to apply Theorem~\ref{th:ppal}:

\begin{theo}
\label{th:wellppal}
  We use notation of Theorem~\ref{th:ppal} and assume that $X$ is smooth. 
Let us also assume that $(G,X,C_\eps,\lambda_\eps)$ is well genericaly finite.

Then, $(G,X,C,\lambda)$ and $(L,C,C_\eps,\lambda_\eps)$ are well genericaly finite.
\end{theo}

\begin{proof}
If $V$ is a vector space endowed with a linear action of a one parameter subgroup $\lambda$ we denote by
$V^\lambda_{<0}$ the set of $v\in V$ such that $\lim_{t\to 0}\lambda(t^{-1})v=0$.\\

Let  $x\in C_\eps$ be such that $T_{\eta_\eps}$  is invertible at $[e:x]$.
Consider  the subtorus $S$ of dimension two containing the images of $\lambda$ and $\lambda_\eps$.
It fixes $x$.
The tangent map of $\eta_\eps$ at the point $[e:x]$ induces a $S$-equivariant linear isomoprhim:
$\theta\,:\,\lieg/\lp_\eps\simeq \lieg^{\lambda_\eps}_{<0}\longto
(T_xX)^{\lambda_\eps}_{<0}$. 
By assumption, $\lieg^{\lambda}_{<0}\subset\lieg^{\lambda_\eps}_{<0}$ and 
$(T_xX)^{\lambda}_{<0}\subset (T_xX)^{\lambda_\eps}_{<0}$.
Since $\theta$ is $S$-equivariant, it follows that it induces an isomorphism
between $\lieg^{\lambda}_{<0}$ and $(T_xX)^{\lambda}_{<0}$.
In particular, $\delta(G,X,C,\lambda)=0$.

Now, the second assertion of Lemma~\ref{lem:iso1} implies that $T_[e:x]\eta$ is invertible.
It follows that $(G,X,C,\lambda)$ is well genericaly finite.

Since $\delta(G,X,C_\eps,\lambda_\eps)=0$, Theorem~\ref{th:ppal} implies that $\delta(L,C,C_\eps,\lambda_\eps)=0$.
Now, Lemma~\ref{lem:iso1} implies that $T_[e:x]\eta_P$ is invertible.
By Lemma~\ref{lem:thetaiso}, it follows that $T_[e:x]\eta_L$ is invertible.
Then, $(L,C,C_\eps,\lambda_\eps)$ is well genericaly finite.
% Conversely, we assume that $(G,X,C,\lambda)$ and $(L,C,C_\eps,\lambda_\eps)$ are well genericaly finite.
% Theorem~\ref{th:ppal} implies that $\eta_\eps$ is genericaly finite. 
% In particular, there exists a $x\in C_\eps^+$ such that $T_{[e:x]}\eta_\eps$ is invertible.

% Set $y=\lim_{t\to 0}\lambda(t)x$ and $z=\lim_{t\to 0}\lambda_\eps(t)y$; we have
% $y\in C\cap C_\eps^+$ and $z\in C_\eps$.
% Since $(G,X,C,\lambda)$ is well genericaly finite, $\lieg_{<0}^\lambda$ is isomorphic to
% $(T_zX)_{<0}^\lambda$ as a $\kk^*$-module for the action of $\lambda$.
% Since  $(L,C,C_\eps,\lambda_\eps)$ is well genericaly finite, $\liel_{<0}^{\lambda_\eps}$ is
% isomorphic to $(T_zC)_{<0}^{\lambda_\eps}$ as a $S$-module. 
% But, we have $\lieg_{<0}^{\lambda_\eps}\simeq
% \lieg_{<0}^\lambda\oplus \liel_{<0}^{\lambda_\eps}$ and 
%   $(T_zX)_{<0}^{\lambda_\eps}\simeq (T_zX)_{<0}^\lambda\oplus (T_zC)_{<0}^{\lambda_\eps}$.
% It follows that $\mu^{{\mathcal D}_\eps}(C,\lambda)=0$. 
 
% Now, \cite[Proposition~5]{GITEigen} implies that $T_[e:y]\eta_\eps$ is invertible.
\end{proof}

\section{Application to Belkale-Kumar's product}

\subsection{An interpretation of coefficient structures}

\paragraphe
Let  $P$ be a parabolic subgroup of the semisimple group $G$.
Let $T\subset B\subset P$ be a maximal torus and a Borel subgroup of $G$.
Let $W$ denote the Weyl group of $T$ and $G$.
For $w\in W$, we set $X(w)=\overline{BwP/P}$, $X^\circ(w)=BwP/P$ and 
denote by $[X(w)]\in {\rm H}^*(G/P,\ZZ)$ the Poincaré dual class of $X(w)$ in cohomology.
Let $w_1,\cdots,w_s\in W$ be such that $\sum_i\codim X(w_i)=\dim G/P$.
Let $c$ be the non negative integer such that 
$$
[X(w_1)].\cdots.[X(w_s)]=c [{\rm pt}].
$$ 

Let $\lambda$ be a one parameter subgroup of $T$ such that $P=P(\lambda)$.
Consider $X=(G/B)^s$ and the following $T$-fixed point $x=(w_1^{-1}B/B,\cdots,w_s^{-1}B/B)$.
Let $C$ denote the irreducible component of $X^\lambda$ containing $x$.
An easy consequence of Kleiman's transversality Theorem (see \cite{Kle1}) is the following lemma 
which express $c$ has a degree.

\begin{lemma}
\label{lem:c=d}
  We have: $\delta(G,X,C,\lambda)=0$ and $c=d(G,X,C,\lambda)$.
\end{lemma}

\begin{proof}
See \cite[proof of Lemma~14]{GITEigen}.
\end{proof}

\paragraphe
Lemma~\ref{lem:c=d} explains how to  express the structure coefficients of ${\rm H}^*(G/P,\ZZ)$ in the basis
of Schubert classes in terms of maps $\eta$'s as in Section~\ref{sec:gen}.
We are now going to discuss Levi-movability, a notion introduced in \cite{BK}:\\

\begin{defin}
 Let $w_i\in W$ such that $\sum_i\codim(X(w_i),X)=\dim(X)$.
Then, $(X(w_1),\cdots,X(w_s))$ is said to be {\it Levi-movable} if for 
generic $l_1,\cdots,l_s\in L$ the intersection 
$l_1X^\circ(w_1)\cap\cdots\cap l_s X^\circ(w_s)$ is transverse at $P/P$.
\end{defin}

\begin{lemma}
\label{lem:Levimovdet}
The following are equivalent:
\begin{enumerate}
\item $(X(w_1),\cdots,X(w_s))$ is Levi-movable;
\item there exists $y\in C$ such that the tangent map $T_{[e:y]}\eta$ of $\eta$ at $[e:y]$ is
invertible. 
\end{enumerate}
\end{lemma}

\begin{proof}
Let $y\in C$ and  $l_1,\cdots,l_s\in L$ such that $y=(l_1w_1^{-1}B/B,\cdots,l_sw_s^{-1}B/B)$.
Since $\eta$ extends the immersion of $C^+$ in $C^+$; the tangent map $T\eta_[e:y]$ restricts 
to the identity on $T_[e:y]C^+$. In particular, it induces a linear map:
$$
\overline{T}\eta_{[e:y]}\,:\,N_{[e:y]}(C^+,G\times_PC^+)\longto N_{y}(C^+,X)
$$  
such that $T_{[e:y]}\eta$ is an isomorphism if and only if $\overline{T}\eta_{[e:y]}$ is.
By $\pi$, $N_{[e:y]}(C^+,G\times_PC^+)$ identifies with $T_eG/P$ that is with $\lieg/\lp$.
Moreover, $N_{y}(C^+,X)$ equals $\bigoplus_iN_{l_iw_i^{-1}B/B}(Pl_iw_i^{-1}B/B,G/B)$
which identifies with $\oplus_i\lieg/(\lp+l_iw_i^{-1}\lb w_il_i)$.
Moreover, after composing by these isomorphisms $\overline{T}\eta_{[e:y]}$ is the canonical map
$\lieg/\lp\longto\oplus_i\lieg/(\lp+l_iw_i^{-1}\lb w_il_i)$.
The lemma follows immediately.
\end{proof}

\subsection{Azad-Barry-Seitz's Theorem}

For later use, we recall in this section the main result of \cite{AzBaSe}.
Let $G$ be a semisimple group and $P$ be a parabolic subgroup of $G$. 
We choose a Levi subgroup $L$ of $P$ and denote by $U$ its unipotent radical.
We are interested in the action of $L$ on the Lie algebra $\lieu$ of $U$.

Let $T$ be a maximal torus of $L$ and $B$ be a Borel subgroup of $G$ containing $T$.
Let $\lieg$ denote the Lie algebra of $G$.
Let $\Delta\subset\Phi^+\subset\Phi$ (resp. $\Delta_L\subset\Phi_L^+\subset \Phi_L$) be 
the set of simple roots, positive roots 
and roots of $G$ (resp. $L$) for $T$ corresponding of $B$ (resp. $B\cap L$).
For any $\alpha\in\Phi$, we denote by $\lieu_\alpha$ the line generated by  the eigenvectors in $\lieg$ of weight $\alpha$.

Since $\lieu$ has no multiplicity for the action of $T$, it has no multiplicity for the action of $L$:
we have a canonical decomposition of $\lieu$ as a sum $\oplus_iV_i$ of irreducible $L$-modules.
Since $T\subset L$, each $V_i$ is a sum of $\lieu_\alpha$ for some $\alpha\in\Phi^+-\Phi^+_L$:
the decomposition $\lieu=\oplus_iV_i$ corresponds to a partition $\Phi^+-\Phi^+_L=\bigsqcup_i\Phi_i$.

Let $\beta$ and $\beta'$ be two positive roots. 
We write 
\begin{eqnarray}
  \label{eq:beta}
\beta=\sum_{\alpha\in \Delta_L}c_\alpha\alpha + \sum_{\alpha\in\Delta-\Delta_L}d_\alpha\alpha,
\end{eqnarray}
with $c_\alpha$ and $d_\alpha$ in $\NN$. 
We also write $\beta'$ in the same way with some $c_\alpha'$ and $d_\alpha'$. 
We write $\beta\equiv\beta'$ if and only if $\sum_{\alpha\in\Delta-\Delta_L}d_\alpha\alpha=
\sum_{\alpha\in\Delta-\Delta_L}d_\alpha'\alpha$. The relation $\equiv$ is obviously an equivalence relation.
Let $S$ denote the set of equivalence classes in  $\Phi^+-\Phi^+_L$ for $\equiv$.
We can now rephrase the main result of \cite{AzBaSe}:

\begin{theo}[Azad-Barry-Seitz]
\label{th:abs}
  For any $s\in S$, $V_s:=\oplus{\alpha\in s}\lieu_\alpha$ is an irreducible $L$-module.
In particular, $\bigsqcup_i\Phi_i$ is the partition in equivalence classes for $\equiv$.
\end{theo}

For any $\alpha\in\Phi$, we denote by $\Ker\alpha$ the Kernel of the character $\alpha$ of $T$.
Let $Z$ be the center of $L$. Let $Z^\circ$ denote the neutral component of $Z$ and $X(Z^\circ)$ denote 
the group of characters of $Z^\circ$.
Under the action of $Z^\circ$, $\lieu$ decompose as 
$$
V=\oplus_{\chi\in X(Z^\circ)}V_\chi,
$$ where $V_\chi$ is the vector subspace 
of weight $\chi$. Since $Z^\circ$ is central in $L$, each $V_\chi$ is $L$-stable. 

Note that $Z^\circ\subset Z\subset T$; and more precisely
$$
Z=\bigcup_{\alpha\in\Delta_L}\Ker\alpha.
$$
It follows that for $\beta$ as in Equation~\ref{eq:beta}, the restriction  $\beta_{|Z^\circ}$ of $\beta$ to $Z^\circ$ equals
$\sum_{\alpha\in\Delta-\Delta_L}d_\alpha\alpha_{|Z^\circ}$.
Moreover, the family $(\alpha_{|Z^\circ})_{\alpha\in\Delta-\Delta_L}$ is free in the rational vector vector space containing 
the characters of the torus $Z^\circ$.
We obtain that 
$$
\beta\equiv\beta'\iff\beta_{|Z^\circ}=\beta'_{|Z^\circ}.
$$
In particular, each $V_s$ is one $V_\chi$ with above notation. In particular, we have:

\begin{coro}
\label{cor:abs}
  Each $V_\chi$ (with $\chi\in X(Z^\circ)$) is an irreducible $L$-module.
\end{coro}

\subsection{A multiplicative formula for structure coefficients of $\kbprod$}

\paragraphe
Let now $Q\subset P$ be two parabolic subgroups of the semisimple group $G$.
Let $T\subset B\subset Q$ be a maximal torus and a Borel subgroup of $G$.

Let $L$ denote the Levi subgroup of  $P$ containing $T$.
Let $W$ (resp. $W_P$) denote the Weyl group of $T$ and $G$ (resp. $L$).

For any $w\in W$, $w^{-1}Bw\cap L$ is a Borel subgroup of $L$ containing $T$.
So, there exists a unique $\wbar\in W_P$ such that 
\begin{eqnarray}
\wbar^{-1}(B\cap L)\wbar=w^{-1}Bw\cap L.
  \label{eq:defwbar}
\end{eqnarray}
To any $w\in W$, we now associated three Schubert varieties in $G/P$, $G/Q$ and $L/L\cap Q$ respectively:
$$
X^{G/P}(w)=\overline{BwP/P},\ \ 
X^{G/Q}(w)=\overline{BwQ/Q}
$$
and
$$
X^{L/L\cap Q}(w)=\overline{(L\cap B)\wbar(L\cap Q/L\cap Q)}.
$$

\begin{theo}
\label{th:kbprod}
Let $w_1,\cdots,w_s\in W$.
We assume that $\sum_i\codim X^{G/Q}(w_i)=\dim G/Q$ and $(X^{G/Q}(w_1),\cdots, X^{G/Q}(w_s))$ is 
Levi-movable.
Then, we have:
\begin{enumerate}
\item \label{ass:kbprod1}
$\sum_i\codim X^{G/P}(w_i)=\dim G/P$ and $\sum_i\codim X^{L/L\cap Q}(w_i)=\dim L/(L\cap Q)$;
\item  \label{ass:kbprod2}
$(X^{G/P}(w_1),\cdots, X^{G/P}(w_s))$ and $(X^{L/L\cap Q}(w_1),\cdots, X^{L/L\cap Q}(w_s))$ are
Levi-movable.
\end{enumerate}
Moreover, by Assertion~\ref{ass:kbprod1} we can define three integers by the formulas:
$$
\begin{array}{rcl}
[X^{G/Q}(w_1)].\cdots. [X^{G/Q}(w_s)]&=&c^{G/Q}_{w_1,\cdots,w_s}[{\rm pt}],\\

[X^{G/P}(w_1)].\cdots. [X^{G/P}(w_s)]&=&c^{G/P}_{w_1,\cdots,w_s}[{\rm pt}] {\rm\ and}\\

[X^{L/L\cap Q}(w_1)].\cdots. [X^{L/L\cap Q}(w_s)]&=&c^{L/L\cap Q}_{w_1,\cdots,w_s}[{\rm pt}].
\end{array}
$$
Then, we have:
$$
c^{G/Q}_{w_1,\cdots,w_s}=c^{G/P}_{w_1,\cdots,w_s}.c^{L/L\cap Q}_{w_1,\cdots,w_s}.
$$
\end{theo}

\begin{proof}
We begin the proof by making some remarks about the tangent space $T_{Q/Q}G/Q$ of $G/Q$ at $Q/Q$.
Let $L_Q$ denote the Levi subgroup of $Q$ containing $T$ and $Z^\circ$ denote its connected center.
We decompose $T_{Q/Q}G/Q$ as a sum $\oplus_{\chi\in X(Z^\circ)} V_\chi$ of eigenvector spaces for the 
action of the torus $Z^\circ$.
Note that $T_{Q/Q}P/Q\subset T_{Q/Q}G/Q$ is stable by the action of $L_Q$. 
Now, Corollary~\ref{cor:abs} implies that there exists $S\subset X(Z^\circ)$ such that  
\begin{eqnarray}
  \label{eq:TPQ}
  T_{Q/Q}P/Q=\oplus_{\chi\in S}V_\chi.
\end{eqnarray}
Let $l\in L_Q$ and $w\in W$. We set $Y^\circ(w)=w^{-1}BwQ/Q$.
One easily checks that $lY^\circ(w)$ is stable by the action of $Z^\circ$.
Since $Q/Q\in lY^\circ(w)$ is a point fixed by $Z^\circ$, $Z^\circ$ acts on $T_{Q/Q}lY^\circ(w)$.
In particular, 
\begin{eqnarray}
  \label{eq:TlY}
T_{Q/Q}lY^\circ(w)=\oplus_{\chi\in X(Z^\circ)}V_\chi\cap T_{Q/Q}lY^\circ(w).
\end{eqnarray}

Since  $(X^{G/Q}(w_1),\cdots, X^{G/Q}(w_s))$ is 
Levi-movable, there exist $l_1,\cdots,l_s\in L_Q$ such that 
\begin{eqnarray}
  \label{eq:sumTlY}
 T_{Q/Q}l_1Y^\circ(w_1^\circ)\oplus\cdots\oplus
  T_{Q/Q}l_sY^\circ(w_s^\circ)= T_{Q/Q}G/Q.
\end{eqnarray}
Consider now the $G$-equivariant projection $\pi\,:\,G/Q\longto G/P$.
Note that the Kernel of the tangent map $T_{Q/Q}\pi$ of $\pi$ at $Q/Q$ is 
$T_{Q/Q}P/Q$.
So, Equations~\ref{eq:TPQ} and \ref{eq:TlY} imply that for any $i=1,\cdots,s$, 
$T_{Q/Q}$ induces an isomorphism from $T_{Q/Q}l_iY^\circ(w_i)\cap\oplus_{\chi\not\in S}V_\chi$ 
onto $T_{Q/Q}l_i\pi(Y^\circ(w_i))$.
Now, Equation~\ref{eq:sumTlY} imply that $\oplus_iT_{P/P}l_i\pi(Y^\circ(w_i))=T_{P/P}G/P$.
Moreover, $L_Q$ is contained in $L$; Assertions~\ref{ass:kbprod1} and \ref{ass:kbprod2} of the theorem 
follows for $G/P$.\\

Recall that $X$ is the variety $(G/B)^s$ and $x=(w_1^{-1}B/B,\cdots,w_s^{-1}B/B)$. 
Let $\lambda$ (resp. $\lambda_\eps$) be a one parameter 
subgroup of $T$ such that $P(\lambda)$ (resp. $P(\lambda_\eps)$) equals $P$ and $Q$.
Let $C$ (resp. $C_\eps$) denote the irreducible component of $X^\lambda$ (resp. $X^{\lambda_\eps}$) 
containing $x$.
With notation of Section~\ref{sec:gen}, Lemma~\ref{lem:c=d} implies that $\delta(G,X,C_\eps,\lambda_\eps)$
and $\delta(G,X,C,\lambda)$ equal zero. Theorem~\ref{th:ppal} implies that $\delta(L,C,C_\eps,\lambda_\eps)=0$.
Assertion~\ref{ass:kbprod1} for $L/L\cap Q$ follows.
Now, the second assertion of Theorem~\ref{th:ppal} with Lemma~\ref{lem:c=d} imply the last formula 
of the theorem.\\

It remains to prove that  $(X^{L/L\cap Q}(w_1),\cdots, X^{L/L\cap Q}(w_s))$ is Levi-movable.
Since $(X^{G/Q}(w_1),\cdots, X^{G/Q}(w_s))$ is Levi-movable, Lemma~\ref{lem:Levimovdet} shows that there exists
$y\in C_\eps$ such that $T_{[e:y]}\eta_\eps$ is invertible.
Now, Lemmas~\ref{lem:iso1} and \ref{lem:thetaiso} imply that $T_{[e:y]}\eta_L$ is invertible.
So,  Lemma~\ref{lem:Levimovdet} allows to conclude.
\end{proof}\\

\begin{remark}
In the case when $G={\rm SL}_n$, Theorem~\ref{th:kbprod} was already obtained in \cite{Richmond:recursion} 
for a lot of pairs $Q\subset P$. 
\end{remark}

\paragraphe
If one know how to compute in $({\rm H}^*(G/P,\ZZ),\kbprod)$ for any maximal $P$ and any $G$,
then  Theorem~\ref{th:kbprod} can be used to compute the structure coefficients of 
$({\rm H}^*(G/Q,\ZZ),\kbprod)$ for
any parabolic subgroup $Q$. 
To illustrate this principle, we state an analogous to \cite[Corollary 23]{Richmond:recursion}:

\begin{coro}
  Let $G={\rm Sp}_{2n}$.
The non-zero coefficients structures of the ring $({\rm H}^*(G/B,\ZZ),\kbprod)$ are all equal to $1$.
\end{coro}

\begin{proof}
The proof proceeds  by induction on $n$.
Let $c$ be a non-zero coefficient structure of $({\rm H}^*(G/B,\ZZ),\kbprod)$.
Let $(w_1,w_2,w_3)$ be such that $[X(w_1)].[X(w_2)].[X(w_3)]=c[{\rm pt}]$.
Since $c$ is non-zero, $(X(w_1),X(w_2),X(w_3))$ is Levi-movable.

Consider the stabilizer $P$ in $G$ of a line in $\kk^{2n}$.
Theorem~\ref{th:kbprod} applied with $B\subset P$ shows that $c$ is the product of coefficient structure
of $({\rm H}^*(G/P,\ZZ),\kbprod)$ and one of $({\rm H}^*({\rm Sp}_{2n-2}/B,\ZZ),\kbprod)$.
The fact that $G/P$ is a projective space and the induction allow to conclude.
\end{proof}

\section{Application to quiver representations}

\subsection{Definitions}
\label{sec:defcarquois}

In this section, we fix some classical notation about quiver representations.

Let $Q$ be a quiver (that is, a finite oriented graph) with vertexes $Q_0$ and arrows $Q_1$.
An arrow $a\in Q_1$ has initial vertex $ia$ and terminal one $ta$.
A representation $R$ of $Q$ is a family $(V(s))_{s\in Q_0}$ of finite dimensional vector spaces and 
a family of linear maps $u(a)\in {\rm Hom}(V(ia),V(ta))$ indexed by $a\in Q_1$.
The dimension vector of $R$ is the family $(\dim(V(s)))_{s\in Q_0}\in \NN^{Q_0}$.

Let us fix $\alpha\in \NN^{Q_0}$ and  a vector space $V(s)$ of dimension $\alpha(s)$ for each $\alpha\in Q_0$. 
Set
$$
\Rep(Q,\alpha)=\bigoplus_{a\in Q_1}{\rm Hom}(V(ia),V(ta)).
$$
Consider also the groups:
$$
\GL(\alpha)\prod_{s\in Q_0}\GL(V(s))
{\rm\ and\ }
\SL(\alpha)\prod_{s\in Q_0}\SL(V(s).
$$
Let $\alpha,\beta\in\ZZ^{Q_0}$. The Ringle form is defined by
$$
\langle\alpha,\beta\rangle=\sum_{s\in Q_0}\alpha(s)\beta(s)-\sum_{a\in Q_1}\alpha(ia)\beta(ta).
$$
Assume now that $\alpha,\beta\in\NN^{Q_0}$.
Following Derksen-Schofield-Weyman (see~\cite{DSW:nbsubrep}), we define $\alpha\circ\beta$ to be the 
number of $\alpha$-dimensional subrepresentation of a general representation of dimension $\alpha+\beta$ 
if it is finite, and $0$ otherwise. 

\subsection{Dominant pairs}

\paragraphe \label{par:defeta}
Let $\lambda$ be a one parameter subgroups of $\GL(\alpha)$.
For any $i\in \ZZ$ and $s\in Q_0$, we set $V_i(s)=\{v\in V(s)\,|\,\lambda(t)v=t^iv\}$ and 
$\alpha_i(s)=\dim V_i(s)$.
Obviously, almost all $\alpha_i$ are zero; and, $\alpha=\sum_{i\in\ZZ}\alpha_i$.
Moreover, $\lambda$ is determined up to conjugacy by the $\alpha_i$'s.

The parabolic subgroup $P(\lambda)$ of $\GL(\alpha)$ associated to $\lambda$ is the set of 
$(g(s))_{s\in Q_0}$ such that for all $i\in\ZZ$ $g(s)(V_i(s))\subset \oplus_{j\leq i}V_j(s)$.

Now, $\Rep(Q,\alpha)^\lambda$ is the set of the $(u(a))_{a\in Q_1}$'s such that 
for any $a\in Q_1$ and for any $i\in\ZZ$, $u(a)(V_i(ia))\subset V_i(ta)$. 
It is  isomorphic to $\prod_i\Rep(Q,\alpha_i)$.
In particular, it is irreducible and denoted by $C$ from now on.

Moreover, $C^+$ is the set of the $(u(a))_{a\in Q_1}$'s such that 
for any $a\in Q_1$ and for any $i\in\ZZ$, $u(a)(V_i(ia))\subset \oplus_{j\leq i}V_j(ta)$.

Consider the morphism $\eta_\lambda\,:\,G\times_{P(\lambda)}C^+\longto \Rep(Q,\alpha)$.
Note that, $P(\lambda)$, $C$ and $C^+$ only depend on the list (ordered by the index $i$) of 
non-zero $\alpha_i$'s. 

\paragraphe The last observation allows the following\\

\begin{defin}
A {\it decomposition} of the vector dimension $\alpha$, is a family $(\beta_1,\cdots,\beta_s)$ of
non-zero vector dimensions such that $\alpha=\beta_1+\cdots+\beta_s$.
We denote the decomposition by $\alpha=\beta_1\pplus\cdots\pplus\beta_s$.
\end{defin}

Any one parameter subgroup $\lambda$ induces a decomposition of $\alpha=\beta_1\pplus\cdots\pplus\beta_s$
where the $\beta_j$'s are the non-zero $\alpha_i$'s ordered by the index $i$.
Note that, up to conjugacy, $P(\lambda)$, $C$ and $C^+$ only depend on this decomposition.
In particular, one can define (up to conjugacy) the {\it map $\eta_{\beta_1\pplus\cdots\pplus\beta_s}$ associated to a decomposition of $\alpha$}.\\

\paragraphe Consider a decomposition $\alpha=\beta_1\pplus\beta_2$ with two dimension-vectors and the associated 
morphism $\eta$. In this section, we collect some easy properties of $\eta$.
Let $(u,v)\in \Rep(Q,\beta_1)\times \Rep(Q,\beta_2)=C\subset \Rep(Q,\alpha)=X$.
Since $\eta$ extends the immersion of $C^+$ in $X$, the tangent map $T_{(u,v)}\eta$ induces the identity
on $T_{(u,v)}C^+$. In particular, it induces a linear map 
$$
\overline{T}\eta_{[e:(u,v)]}\,:\,N_{[e:(u,v)]}(C^+,G\times_PC^+)\longto N_{(u,v)}(C^+,\Rep(Q,\alpha)).
$$
Moreover, $N_{[e:(u,v)]}(C^+,G\times_PC^+)$ identifies with 
$\oplus_{s\in Q_0}{\rm Hom}(V(s),W(s))$ and $N_{(u,v)}(C^+,\Rep(Q,\alpha))$ with
$\oplus_{a\in Q_1}{\rm Hom}(V(ia),W(ta))$.
A direct computation gives the following

\begin{lemma}
\label{lem:Teta}
  With the above identification, we have:
$$
\overline{T}\eta_{[e:(u,v)]}(\sum_{s\in Q_0}\varphi(s))=\sum_{a\in Q_1}v(a)\varphi(ta)-\varphi(ha)u(a).
$$
In particular, the Kernel of $\overline{T}\eta_{(u,v)}$ is ${\rm Hom}(u,v)$ and its Image is
${\rm Ext}(u,v)$. 
\end{lemma}

The quantities $\delta(\eta)$ and $d(\eta)$ are also particularly interesting:

\begin{lemma}
\label{lem:deltad2}
Consider a decomposition $\alpha=\beta_1\pplus\beta_2$ and the associated map $\eta$.
Then,
\begin{enumerate}
\item $\delta(\eta)=-\langle \beta_1,\beta_2\rangle$, and
\item $d(\eta)=\beta_1\circ \beta_2$.
\end{enumerate}
\end{lemma}

\begin{proof}
  By the discussion preceding Lemma~\ref{lem:Teta}, $\delta(\eta)$ equals the difference between the dimension 
of $\oplus_{a\in Q_1}{\rm Hom}(V(ia),W(ta))$ and of $\oplus_{s\in Q_0}{\rm Hom}(V(s),W(s))$. 
The first assertion follows.

Let $u\in\Rep(Q,\alpha)$.
Using Immersion~\ref{eq:immersion}, one identifies the fiber $\eta^{-1}(u)$ with the set $u$-stable subspaces
of $V$ of dimension $\beta_1$. In particular, $\eta^{-1}(u)$ identifies with the set of $\beta_1$-dimensional
subrepresentations of $u$. Since the characteristic of $k$ is assumed to be zero, when $u$ is generic this numbers
equals $d(\eta)$ on one hand and $\beta_1\circ\beta_2$ on the other one. 
\end{proof}\\

If $Y$ is a smooth variety of dimension $n$, $\Tau Y$ denotes its tangent bundle.
The line bundle $\bigwedge^n\Tau Y$ over $Y$ will be called the {\it determinant bundle} and
denoted by $\Det Y$.
If $\varphi\::\:Y\longto Y'$ is a morphism between smooth variety, we denote
by $\Det\varphi\::\:\Det Y\longto\Det Y'$ the determinant of its tangent map $T\varphi$.
We consider now the restriction of  $\Det\eta$ to $C^+$: it is a $P$-invariant section of
the $P$-linearized line bundle $\Det$ over $C^+$ defined by 
$\Det=\Det(G\times_P C^+)_{|C^+}^*\otimes\Det(X)_{|C^+}$.

Recall that for any $s\in Q_0$, we have fixed a vector space $V(s)$ of dimension $\alpha(s)$.
Let us fix, for any $s\in Q_0$ a decomposition $V(s)=V_1(s)\oplus V_2(s)$ such that $\dim V_i(s)=\beta_i(s)$
for $i=1,\,2$.
Consider the one parameter subgroup $\lambda$ of $\GL(\alpha)$ defined by $\lambda(s)(t)$ stabilizes the 
decompostion $V_1(s)\oplus V_2(s)$, equals to $\Id$ when restricted to $V_1(s)$ and 
$t\Id$ when restrected to $V_2(s)$.
It satisfies $P(\lambda)=P$, $\Rep(Q,\alpha)^\lambda=C$ and $C^+(\lambda)=C^+$.

\begin{lemma}
\label{lem:mudetC}
We assume that $\langle\beta_1,\beta_2\rangle=0$.
  The one parameter subgroup $\lambda$ acts trivially on $\Det_{|C}$.
\end{lemma}

\begin{proof}
  Since $C$ is an affine space, $\lambda$ acts by the same character on each fiber of $\Det_{|C}$.
Since $\eta$  extend the identity on $C^+$, its character is the difference between the weights of $\lambda$
in $$
N_0(C^+,X)\simeq\oplus_{a\in Q_1}{\rm Hom}(V_1(ia),V_2(ta))
$$ and in 
$$
N_0(C^+,G\times_PC^+)\simeq T_eG/P\simeq\oplus_{s\in Q_0}\Hom(V_1(s),V_2(s)). 
$$
So, this character equals:
$$
\sum_{a\in Q_1}\beta_1(ia)\beta_2(ta)-\sum_{s\in Q_0}\beta_1(s)\beta_2(s);
$$ 
that is, $-\langle\beta_1,\beta_2\rangle$. The lemma follows.
\end{proof}

\subsection{Two formulas for $d(\eta_{\beta_1\pplus\cdots\pplus\beta_s})$}

Here comes the main result of this section:

\begin{theo}
\label{th:carquois}
  Let $\alpha=\beta_1\pplus\cdots\pplus\beta_s$ be a decomposition of $\alpha$ such that
for all $i<j$, $\langle \beta_i,\beta_j\rangle=0$. 

Then, $\delta(\eta_{\beta_1\pplus\cdots\pplus\beta_s})=0$ and 
$$\begin{array}{rl}
  d(\eta_{\beta_1\pplus\cdots\pplus\beta_s})&=(\beta_1\circ \alpha-\beta_1).(\beta_2\circ\alpha-\beta_1-\beta_2).\cdots.
(\beta_{s-1}\circ\beta_s)\\
&=(\alpha-\beta_s\circ\beta_s).(\beta-\beta_s-\beta_{s-1}\circ\beta_{s-1}).\cdots.
(\beta_1\circ\beta_2).
\end{array}
$$
\end{theo}

\begin{proof}
 By Section~\ref{par:defeta}, the codimension of $C^+$ in $G\times_PC^+$ is
$$
\sum_{i<j}\sum_{s\in Q_0}\beta_i(s)\beta_j(s);
$$
and, the codimension of $C^+$ in $\Rep(Q,\alpha)$ is
$$
\sum_{i<j}\sum_{a\in Q_1}\beta_i(ia)\beta_j(ta).
$$
Since $\forall i<j$  $\langle \beta_i,\beta_j\rangle=0$, this implies 
that $\delta(\eta_{\beta_1\pplus\cdots\pplus\beta_s})=0$.

We will just prove the first formula for $d(\eta_{\beta_1\pplus\cdots\pplus\beta_s})$. 
The second one can be proved in a similar way.
When $s=2$, the theorem follows from Lemma~\ref{lem:deltad2}.
Assume that $s=3$. A direct application of Theorem~\ref{th:ppal} with
 $\eta_\eps=\eta_{\beta_1\pplus\beta_2\pplus\beta_3}$ and $\eta=\eta_{\beta_1\pplus(\alpha-\beta_1)}$ gives
$$
\begin{array}{rl}
  d(\eta_{\beta_1\pplus\beta_2\pplus\beta_3})&=(\beta_1\circ\alpha-\beta_1).d(\eta_{\beta_2\pplus\beta_3})\\
&=(\beta_1\circ\alpha-\beta_1).(\beta_2\circ\beta_3).
\end{array}
$$
One can easily ends the proof by an induction on $s$.
\end{proof}\\

\begin{remark}
In the proof of Theorem~\ref{th:carquois}, the induction was made by the paranthésages 
$\beta_1\pplus\cdots\pplus\beta_s=\beta_1\pplus(\beta_2\pplus\cdots\pplus\beta_s)$ and
$\beta_1\pplus\cdots\pplus\beta_s=(\beta_1\pplus(\beta_2\cdots\pplus\beta_s)$.
All other paranthésage gives a similar formula.  
\end{remark}

\paragraphe
We now want to discuss the assumption ``$\forall i<j$  $\langle \beta_i,\beta_j\rangle=0$''.
This assumption is actually similar to Levi-movability. Indeed, we have the equivalent of
 Lemma~\ref{lem:Levimovdet}:

 \begin{lemma}
\label{lem:wellcarquois}
 Let $\alpha=\beta_1\pplus\cdots\pplus\beta_s$ be a decomposition of $\alpha$ such that
$\delta(\eta_{\beta_1\pplus\cdots\pplus\beta_s})=0$. Then, the following are equivalent:
\begin{enumerate}
\item \label{ass1:wellcarquois}
for all $i<j$, $\langle \beta_i,\beta_j\rangle=0$
and $d(\eta_{\beta_1\pplus\cdots\pplus\beta_s})\neq 0$; 
\item 
\label{ass2:wellcarquois}
there exists $y\in C$ such that the tangent map of $\eta_{\beta_1\pplus\cdots\pplus\beta_s}$ 
at $[e:y]$ is invertible.
\end{enumerate}
\end{lemma}

 \begin{proof}
   Let $\underline{V}=\oplus_i\underline{V}_i$ be a decomposition of $\underline{V}$ such that 
$\dim\underline{V}_i=\beta_i$.
Consider the linear action of the torus $Z={\mathbb G}_m^s$ on $\underline{V}$ such that
$(t_1,\cdots,t_s).v=t_iv$ for all $t_i\in{\mathbb G}_m$ and $v\in \underline{V}_i(s)$ for any 
$s\in Q_0$. Since $Z$ is embedded in $\GL(\alpha)$ it also acts on $G\times_PC^+$.

Let $y$ be a point in $C$ satisfying Assertion~\ref{ass2:wellcarquois}.
Since, $Z$ fixes $[e:y]$ and $\eta$ is $G$-equivariant, $T\eta_{\beta_1\pplus\cdots\pplus\beta_s}$
is $Z$-equivariant for the tangent action of $Z$.
It follows that for all $i<j$, $T\eta_{\beta_1\pplus\cdots\pplus\beta_s}$ induces an isomorphism
from the eigensubspaces of $T_{[e:y]}G\times_PC^+$ and $T_y\Rep(Q,\alpha)$ of weight $t_jt_i^{-1}$.
In particular, these two eigensubspaces have the same dimension. But, a direct compution shows 
that the difference between these two dimension is precisely $\langle \beta_i,\beta_j\rangle$.
Assertion~\ref{ass1:wellcarquois} follows.\\

Conversely, let us assume that Assertion~\ref{ass1:wellcarquois} follows.
Since $d(\eta_{\beta_1\pplus\cdots\pplus\beta_s})\neq 0$, there exists a point $G\times_PC^+$ 
where the tangent map of $\eta_{\beta_1\pplus\cdots\pplus\beta_s}$ is invertible.
Since $\eta$ is $G$-equivariant, its determinant is not identicaly zero on  $C^+$.
Using the fact for all $i<j$ $\langle \beta_i,\beta_j\rangle=0$, a direct computation
(like in the proof of Lemma~\ref{lem:mudetC}) shows that $Z$ acts trivialy on 
$\Det_{|C}$. By \cite[Proposition 5]{GITEigen},  the determinant of
$\eta$ is not identicaly zero on $C$. Assertion~\ref{ass2:wellcarquois} follows.
 \end{proof}

\paragraphe
The dimension of ${\rm Ext(u,v)}$ for generic $\alpha$ and $\beta$ dimensional 
representations $u$ and $v$ is denoted by $\ext(\alpha,\beta)$.

\begin{coro}
We assume that $Q$ has no oriented cycle.
  Let $\alpha,\,\beta$ and $\gamma$ be three dimension-vectors. 
We assume that $\langle\alpha,\beta\rangle=\langle\alpha,\gamma\rangle=0$ and 
$\beta\circ\gamma=1$.

Then, $\alpha\circ \beta+\gamma=(\alpha\circ\beta).(\alpha\circ \gamma)$.
\end{coro}

\begin{proof}
Theorem~\ref{th:carquois} applied to $\alpha\pplus\beta\pplus\gamma$ gives:  
$$
(\alpha+\beta\circ \gamma).(\alpha\circ\beta)=
(\alpha\circ \beta+\gamma).(\beta\circ\gamma)=
(\alpha\circ \beta+\gamma),
$$
since $(\beta\circ\gamma)=1$.
If $\alpha\circ\beta=0$, the corollary follows.
Now assume that $\alpha\circ\beta\neq 0$.
Lemma~\ref{lem:wellcarquois} implies that the determinant of
$\eta_{\alpha\pplus\beta}$  is not identically zero on $C$.
But, Lemma~\ref{lem:Teta} implies that $\ext(\alpha,\beta)=0$.
Now, the corollary is a direct consequence of Lemma~\ref{lem:arondb=1} below.
\end{proof}

\begin{lemma}
\label{lem:arondb=1}
We assume that $Q$ has no oriented cycle.
  Let $\alpha,\,\beta$ and $\gamma$ be three vector dimensions.
We assume that $\beta\circ\gamma=1$ and  $\ext(\alpha,\beta)=0$.

Then, $\alpha+\beta\circ\gamma=\alpha\circ\gamma$.
\end{lemma}

\begin{proof}
In \cite{DSW:nbsubrep}, Derksen-Schofield-Weyman prove that $\alpha\circ\gamma$
equals the dimension of $\kk[\Rep(Q,\gamma)]_{\langle\alpha,\cdot\rangle}$. 
Consider the multiplication morphism:
$$
m\,:\,\kk[\Rep(Q,\gamma)]_{\langle\alpha,\cdot\rangle}\otimes
\kk[\Rep(Q,\gamma)]_{\langle\beta,\cdot\rangle}\longto
\kk[\Rep(Q,\gamma)]_{\langle\alpha+\beta,\cdot\rangle}.
$$ 
We claim that $m$ is an isomorphism. The lemma will follow directly.
Since $\dim(\kk[\Rep(Q,\gamma)]_{\langle\beta,\cdot\rangle})=1$ and
$\kk[\Rep(Q,\gamma)]$ has no zero-divisors, $m$ is injective.

Since $\ext(\alpha,\beta)=0$, $\eta_{\alpha\pplus\beta}$ is dominant. 
But, it is proper; so, it is surjective.

In \cite{DW:saturation}, Derksen-Weyman prove that 
$\kk[\Rep(Q,\gamma)]_{\langle\alpha+\beta,\cdot\rangle}$ is generated by functions $c^V$ 
associated  to various  $\alpha+\beta$-dimensional representation $V$ (see also \cite{DomZub}).
Since  $\eta_{\alpha\pplus\beta}$ is surjective, there exists an $\alpha$-dimensional 
subrepresentation $V'$ of $V$. 
By \cite[Lemma~1]{DW:saturation}, $c^V=c^{V'}.c^{V/V'}$.
It follows that $m$ is surjective. 
\end{proof}

\bibliographystyle{amsalpha}
\bibliography{biblio}

\begin{center}
  -\hspace{1em}$\diamondsuit$\hspace{1em}-
\end{center}

\vspace{5mm}
\begin{flushleft}
N. R.\\
Universit{\'e} Montpellier II\\
D{\'e}partement de Math{\'e}matiques\\
Case courrier 051-Place Eug{\`e}ne Bataillon\\
34095 Montpellier Cedex 5\\
France\\
e-mail:~{\tt ressayre@math.univ-montp2.fr}  
\end{flushleft}

\end{document}